\documentclass[12pt]{article}
\usepackage[latin1]{inputenc}
\usepackage{amssymb}
\usepackage{amsmath}
\usepackage{amsthm}

\usepackage{graphicx, color}
\usepackage{epstopdf}

\newcommand{\F}{\mathbb{F}}
\newcommand{\R}{\mathbb{R}}

\newcommand{\Z}{\mathbb Z}

\theoremstyle{definition}

\newtheorem{thm}{Theorem}[section]

\newtheorem{prop}[thm]{Proposition}

\newtheorem{lem}[thm]{Lemma}
 
\newtheorem{rem}[thm]{Remark}

\newtheorem{defi}[thm]{Definition}

\newcommand{\vs}{\vspace{0.3cm}}

\date{}
\author{}

\begin{document}

\title{Spaces of orders of some one-relator groups}
\author{Juan Alonso, Joaqu\'\i n Brum}
\maketitle

\begin{abstract}
We show that certain left orderable groups admit no isolated left orders. The groups we consider are cyclic amalgamations of a free group with a general left orderable group, the HNN extensions of free groups over cyclic subgroups, and a particular class of one-relator groups. In order to prove the results about orders, we develop perturbation techniques for actions of these groups on the line.  
\end{abstract}

\section{Introduction}\label{intro}


A group $G$ is said to be left orderable if it admits a total order invariant by left translations. Left orderability of groups is a wide and active topic of research (see \cite{ghys},\cite{clay rolfsen},\cite{GOD},\cite{KM}). Within this theory, an important object of study is the {\em space of left orders} $\mathcal{LO}(G)$ of a left orderable group $G$. This is the set of left orders on $G$ endowed with a natural topology that makes it a Hausdorff, totally disconnected and compact space \cite{sikora}. For a countable group $G$, this implies that $\mathcal{LO}(G)$ is a Cantor set exactly when it contains no isolated orders. In addition to that, Linnell showed that this space is either finite or uncountable \cite{linnell}, and Tararin classified the groups that have finitely many left orders \cite[Theorem 5.2.1]{KM}. In light of these results, one of the main interesting problems concerning the topology of the space of left orders is to determine which left orderable groups admit isolated orders. 

This problem turns out to be a complex one, as shown by the partial results that have been obtained. For virtually solvable groups, Rivas and Tessera gave a complete description of this spaces showing that they are either finite or Cantor sets \cite{rivas tessera}. Free products of left orderable groups admit no isolated orders \cite{rivas_free} and neither does cyclic amalgams of free groups \cite{abr}. On the other hand, $\F_n\times\Z$ admits isolated orders if and only if $n$ is even \cite{ping}. 

Here we provide a generalization of the results about amalgams in \cite{abr}.      

\begin{thm}\label{manija} Let $G=\F_n*_{\Z}H$ be a left orderable group with $n\geq 2$. Then $G$ has no isolated orders. 
\end{thm}

The orderability of $G$ in Theorem \ref{manija} is equivalent to that of $H$ since the amalgamating subgroup is cyclic, as shown in \cite{BG camb}. The hypothesis that $n\geq 2$ is necessary, for example, the torus knot groups $\langle a,b|a^m=b^k\rangle$ have isolated orders \cite{navas},\cite{ito}. 

The situation for general amalgamated products is more complex, even when the amalgamating subgroups are cyclic. Ito \cite{ito2} constructs a wide class of amalgamated products of the form $G*_{\Z}H$ that do have isolated orders. The groups $G$ and $H$ considered in \cite{ito2} both have isolated orders. Theorem \ref{manija} points in the other direction, for amalgams $G*_{\Z}H$ where one of the factors is a non abelian free group.  

The results and techniques used in \cite{abr} suggest the question:

\vs
{\bf Question 1:} Can a one-relator group generated by three or more elements have an isolated order?
\vs

We obtain some partial results in this direction. The next one can also be seen as a generalization of Theorem 1.1 in \cite{abr}, this time dealing with the case of an HNN extension.

\begin{thm}\label{hnn} Let $G=\langle t,x_1,...,x_n|tw_1t^{-1}=w_2\rangle$ with $n\geq 2$ and non trivial $w_1,w_2\in \langle x_1,...,x_n \rangle$. Then $G$ has no isolated orders. 
\end{thm}

One-relator groups are left orderable unless they have torsion \cite{brodskii}, and they only have torsion when the relation is a proper power. Thus the groups in Theorem \ref{hnn} are left orderable. Again, the result does not hold for $n=1$. This can be seen for the Klein bottle group $\langle a,b|aba^{-1}=b^{-1}\rangle$, that has finitely many orders \cite{GOD}. 

Following Question 1, we consider one-relator groups with more complex relations. The techniques we developed seem well adjusted to the case where the relation contains only positive powers of some given generator $t$. Namely, $G=\langle t,x_1,...,x_n|tw_1\cdots tw_k=1\rangle$ for $n\geq 2$ and $w_i\in\langle x_1,...,x_n\rangle$. If $k=1$ then $G$ is free, and for $k=2$ we can change the presentation to $G=\langle s,x_1,...,x_n| s^2w_1^{-1}w_2=1\rangle$ that is an amalgam covered in Theorem \ref{manija}. Our techniques allow us to obtain the case $k=3$, that turns out to be quite non trivial.

\begin{thm}\label{positiveword} Let $G=\langle t,x_1,...,x_n|tw_1tw_2tw_3=1\rangle$ with $n\geq 2$ and $w_i\in\langle x_1,...,x_n\rangle$ for $i=1,2,3$. Assume that $w_i\neq w_j$ for some $i,j$. Then $G$ has no isolated orders. 
\end{thm} 
\begin{rem} If $w_1=w_2=w_3$ then the group $G$ in Theorem \ref{positiveword} has torsion and therefore is not left orderable. 
\end{rem}

The conclusion of Theorem \ref{positiveword} does not hold when $n=1$, as Proposition 8.1 in \cite{dehornoy} shows that the groups $\langle t,x|tx^ptx^ptx^{-1}\rangle$ for $p\geq 1$ have isolated left orders.

We would conjecture that the groups of the form $G=\langle t,x_1,...,x_n|tw_1\cdots tw_k=1\rangle$ with $n\geq 2$ do not have isolated orders in general. However, we meet technical obstructions in our approach when $k\geq 4$.

The method for proving these theorems involve the close relationship between left orders and actions on the line. A countable group $G$ is left orderable if and only if it admits a faithful action by orientation preserving homeomorphisms of the line. Indeed, a left order on $G$ induces an action on the line via the construction called {\em dynamical realization} (see Section \ref{preliminaresordenes}). Furthermore, the topology of $\mathcal{LO}(G)$ is related to rigidity properties on the space of such actions. More precisely, an isolated order induces an action that is {\em structurally stable} (also called {\em rigid}), see \cite{abr} for the definition. In \cite{abr} it is shown that cyclic amalgams of free groups have no rigid actions on the line, implying they have no isolated orders. Here we do not deal with rigid actions in order to avoid more technicalities, instead we use a weaker version of this principle, namely Proposition \ref{aislado}.

The proofs of our three theorems follow the same rough strategy. We start from the dynamical realization of a given order and construct an arbitrarily small perturbation that has non trivial stabilizer on the orbit of 0, then we use Proposition \ref{aislado} to conclude. The groups in our theorems contain free subgroups, namely the factor $\F_n$ in Theorem \ref{manija}, and the subgroups generated by $x_1,\ldots,x_n$ in Theorems \ref{hnn} and \ref{positiveword} (see Freiheitsatz Theorem in \cite{Baumslag}), and the strategy in all cases is to perturb the action restricted to these free subgroups in a way that can be extended to an action of the whole group. Our main technical tool to achieve that is Lemma \ref{tronco}, that gives a way to perturb actions of free groups creating non trivial stabilizers, while controlling the behaviour of a particular element and of finitely many partial orbits.  

We would like to point out that Lemma \ref{tronco2} is important on its own right. Let $Rep(G,H)$ denote the set of of group representations of $G$ in $H$.  For each $w\in\F_n$, and any group $H$, we can define a  {\em word map} from  $Rep(\F_n,H)$ to $H$ that associates a representation $\rho$ to $\rho(w)\in H$. Lemma \ref{tronco2} says that when $H=Homeo_{+}(\R)$, the group of orientation preserving homeomorphisms of the real line, the word map is surjective for any non trivial $w\in\F_n$.
\vs

{\bf Acknowledgements.} The authors want to thank Crist\'obal Rivas for his invitation to Santiago de Chile and for several fruitful conversations, and the referees for many useful suggestions.

\section{Preliminaries}


\subsection{Left orders and actions on the line}\label{preliminaresordenes}

A left order on a group $G$ is a total order $<$ satisfying  that given $f,g,h\in G$ such that $f < h $ then $ gf < gh$. If $G$ admits a left order we say that $G$ is {\em left orderable}. The reader unfamiliar with this notion may wish to consult \cite{clay rolfsen}, \cite{GOD}, \cite{KM}. The groups under consideration in this paper are left orderable as we already mentioned in section \ref{intro}.

A natural topology can be defined on the set $\mathcal{LO}(G)$ of all left orders on $G$, making it a compact and totally disconnected space. A local base at a left order $<$ is given by the sets \[V_{g_1,\ldots,g_n}:=\{<' \in\mathcal{LO}(G)\mid 1 <'g_i\},\]
where $\{g_1,\ldots,g_n\}$ runs over all finite subsets of $<$-positive elements of $G$. In particular, a left order $<$ is isolated in $\mathcal{LO}(G)$ if and only if there is a finite set $S\subset G$ such that $<$ is the only left order satisfying \[id < s\, , \text{ for every $s\in S$}.\]
When the group is countable this topology is metrizable \cite{clay rolfsen}, \cite{GOD},\cite{sikora}. For instance, if $G$ is finitely generated, and $B_n$ denotes the ball of radius $n$ with respect to a finite generating set, then we can declare that $dist(<_1,<_2)=1/n$, if $B_n$ is the largest ball in which $<_1$ and $<_2$ coincide.

When the group $G$ is countable, for every left order $<$ on $G$, one can attach a fixed-point-free action $\rho: G\to Homeo_+(\R)$ that models the left translation action of $G$ on $(G,<)$, in the sense that
\begin{equation} f< g \Leftrightarrow \rho(f)(0)< \rho(g)(0). \label{eq dyn real}\end{equation}
This is the so called, {\em dynamical realization} of $<$ (which is unique up to conjugation), and $0$ is sometimes called the {\em base point}, see \cite{clay rolfsen}, \cite{GOD}, \cite{ghys}.

On the other hand, any representation $\rho: G\to Homeo_+(\R)$ defines a {\em partial} left invariant order on $G$ through equation \ref{eq dyn real}. This is a total order exactly when the stabilizer of $0$ under $\rho$ is trivial.

Given a group $G$ we consider the set $Rep(G,Homeo_+(\R))$ of group representations from $G$ to $Homeo_+(\R)$ endowed with the pointwise convergence. That is, $\rho_n$ converges to $\rho$ if and only if $\rho_n(g)$ converges to $\rho(g)$ for all $g\in G$, where the convergence $\rho_n(g)\to\rho(g)$ is given by the compact open topology: for every positive $\varepsilon$ and for every compact set $K\subset M$ there is $n_0$ such that $n\geq n_0$ implies  \[\sup_{x\in K} |\rho_n(g)(x)-\rho(g)(x)|\leq \varepsilon.\]

\begin{rem} \label{separability} Observe that the convergence of $\rho_n\to \rho$ in $Rep(G, Homeo_+(\R))$ is equivalent to require that $\rho_n(g)\to\rho(g)$ for every $g$ in a generating set of $G$. 
\end{rem}

The next result is a way to relate the topology of $\mathcal{LO}(G)$ with that of $Rep(G,Homeo_{+}(\R))$, and will be the key tool to prove all our theorems.

\begin{prop} \label{aislado} Let $\rho\in Rep(G,Homeo_{+}(\R))$ be the dynamical realization of a total left order $<$ on $G$ (in the sense of equation \ref{eq dyn real}). If $\rho$ can be arbitrarily approximated by representations that have non trivial stabilizers on the orbit of $0$, then $<$ is not isolated in $\mathcal{LO}(G)$. 
\end{prop}
\begin{proof} Let $F\subseteq G$ be a finite set with $1<g$ for every $g\in F$. We take a neighbourhood $V$ of $\rho$ so that if $\rho'\in V$ then $0<\rho'(g)(0)$ for $g\in F$. By our hypothesis, there exists $\rho'\in V$ that has non trivial stabilizer on the orbit of $0$. This induces a partial left order $\prec$ on $G$. Since $H=Stab_{\rho'}(0)$ is a subgroup of a left orderable group, it is also left orderable and has at least two different left orders. By the convex extension procedure (see \cite[\S2.1]{GOD}) we can extend $\prec$ to at least two different total left orders $<_1$ and $<_2$ on $G$, satisfying $1 <_i g$ for $g\in F$ and $i=1,2$. One of them must be different from $<$.
\end{proof}

\subsection{Conjugacy and roots in $Homeo_{+}(\R)$}

Here we present some facts and constructions on line homeomorphisms that will be needed in the sequel. Given $\phi\in Homeo_{+}(\R)$ we define the following sets:
\[\begin{aligned}
    Fix(\phi)=\{x\mid \phi(x)=x\} \\
    Inc(\phi)=\{x\mid \phi(x)>x\} \\
    Decr(\phi)=\{x\mid \phi(x)<x\}
  \end{aligned}\]
  
These sets help us study the conjugacy class of $\phi$. If $\psi\phi_1\psi^{-1}=\phi_2$, then $\psi$ induces bijections between the corresponding sets for $\phi_1$ and $\phi_2$. (Namely, $\psi(Fix(\phi_1))=Fix(\phi_2)$ and so on). On the other hand, two homeomorphisms $\phi_1$ and $\phi_2$ are conjugated in $Homeo_{+}(\R)$ if there exists $\psi\in Homeo_+(\R)$ that maps $Fix(\phi_1)$ to $Fix(\phi_2)$ and $Inc(\phi_1)$ to $Inc(\phi_2)$ (and so maps $Decr(\phi_1)$ to $Decr(\phi_2)$). With this in mind, we define a {\em weak conjugation} as follows.

\begin{defi}\label{weakdefi} For $\psi,\phi_1, \phi_2$ homeomorphisms of the real line, we will say that $\psi$ is a {\em weak-conjugation} from $\phi_1$ to $\phi_2$ if \begin{itemize}
\item  $\psi(Fix(\phi_1))=Fix(\phi_2)$ and
\item $\psi(Inc(\phi_1))=Inc(\phi_2)$. \end{itemize}
Additionally, if for an interval $I$ we have that $\psi\phi_1(x)=\phi_2\psi(x)$ for all $x\in I$ we will say that the weak conjugation $\psi$ is {\em strong} on $I$.
\end{defi}

Observe that conjugacy and weak-conjugacy classes are identical, but it is much easier to find/build weak conjugations rather than true conjugating elements. We will need a result that allows us to pass from a weak conjugation to a conjugation, while respecting the parts in which the weak conjugation is strong. This is Lemma 2.7 in \cite{abr}, that we state below.

\begin{lem}\label{promotion} {\em Let $\psi,\phi_1,\phi_2\in Homeo_+(\R)$. If $\psi$ is a weak-conjugation from $\phi_1$ to $\phi_2$ that is strong on a interval $I$, then there exists a conjugation $\bar{\psi}$ from $\phi_1$ to $\phi_2$ such that:
\begin{itemize}
\item $\bar{\psi}(x)=\psi(x)$ for every $x\in I$ and
\item $\bar{\psi}(x)=\psi(x)$ for every $x\in Fix(\phi_1)$.
\end{itemize}
Moreover, $\bar \psi$ agrees with $\psi$ over $I\cup \phi_1(I)$.}

\end{lem}

We refer the proof to \cite{abr}.

\begin{rem}\label{paintervalo} We will use Lemma \ref{promotion} in a slightly stronger form, where $\phi_1$ and $\phi_2$ are homeomorphisms of arbitrary intervals $I$ and $J$ respectively. The map $\psi:I\to J$ is a weak conjugation following a straightforward adaptation of Definition \ref{weakdefi}. This stronger version of Lemma \ref{promotion} is obtained as a corollary through conjugation. 
\end{rem}

We will also need to take square roots of homeomorphisms under composition.

\begin{lem}\label{raiz} Every $h\in Homeo_{+}(\R)$ has a square root.
\end{lem}

\begin{proof} Translations on $\R$ clearly have square roots. For $h\in Homeo_{+}(\R)$ define $\psi$ so that $Fix(\psi)=Fix(h)$, and if $I$ is a connected component of $\R-Fix(h)$ then $\psi|_{I}$ is a square root of $h|_{I}$, which exists because $h|_{I}$ is conjugated to a translation. It is easy to check that $\psi^2=h$. \end{proof}

\begin{rem} \label{mas raiz} In the proof above it is clear that if $q\in Fix(h)$ and $\psi_0$ is a square root of $h|_{(-\infty,q)}$, we can choose $\psi$ as an extension of $\psi_0$. We can also adapt Lemma \ref{raiz} for a homeomorphism $h:(-\infty,q_1]\to(-\infty,q_2]$, obtaining $\psi:(-\infty,q_3]\to(-\infty,q_4]$ with $\psi^2$ a restriction of $h$. This can be done by extending $h$ to a homeomorphism of $\R$ and then restricting $\psi$ to a suitable interval. 
\end{rem}

\section{Key technical tools}\label{troncocomun}


Let $w\in\F_n=\langle x_1,...,x_n\rangle$ be a reduced word, and write $w=a_m \cdots a_1$ with $a_j\in\{x_1^{\pm 1},...,x_n^{\pm 1}\}$. We define $w_0=e$ and $w_j=a_{j}...a_{1}$ for $0<j\leq m$.
If $\rho\in Rep(\F_n,Homeo_{+}(\R))$ and $x\in\R$ we will be interested in the sequence $$S(\rho,w,x)=(\rho(w_0)(x),...,\rho(w_m)(x))$$ that we call the {\em trajectory} of $x$ by $w$ under the action $\rho$. 

Trajectories will play a key role in our perturbation techniques. On one hand, we will construct new representations by ``extending'' pre-fixed trajectories. That is, we will first define some arbitrary, but suitable, sequence $S=(s_0,...,s_m)\in\R^{m+1}$ and then find a representation that realizes $S$ as a trajectory by $w$. 

On the other hand, we will need to perturb the representations while keeping in mind the effect on certain trajectories. To do that, for each generator $x_i$ we will need to look at the minimum point from which we can perturb the map $\rho(x_i)$ without changing the trajectory $S(\rho,w,x)$. That point is the largest one where we apply the generator $x_i$ in the trajectory. 

Taking into account that we will work with ``trajectories'' {\em before} realizing them by representations, it will be useful to make the relevant definitions in a combinatorial context, without reference to a specific representation.   

\begin{defi}\label{trajectory} Take $w\in\F_n=\langle x_1,...,x_n\rangle$ with $|w|=m$, a sequence $S=(s_0,...,s_m)\in\R^{m+1}$ and $i\in\{1,...,n\}$. Write $w=a_m \cdots a_1$ with $a_j\in\{x_1^{\pm 1},...,x_n^{\pm 1}\}$. We define

$$D_w(S,i)=\{s_j:a_{j+1}=x_i \ or \ a_{j}=x_i^{-1}\}$$
and
$$d_w(S,i) = \max D_w(S,i)$$ 
When $x_i$ does not appear in $w$, we set $D_w(S,i)=\emptyset$ and $d_w(S,i)=-\infty$. 
  
\end{defi}

We remark that when $S=S(\rho,w,x)$ is an actual trajectory, the set $D_w(S,i)$ is the subset of $S$ where we apply the generator $x_i$ in the trajectory. Look at Figure \ref{f1} for an example.



\begin{figure}[h!] 
\centering
\includegraphics[width=0.8\textwidth]{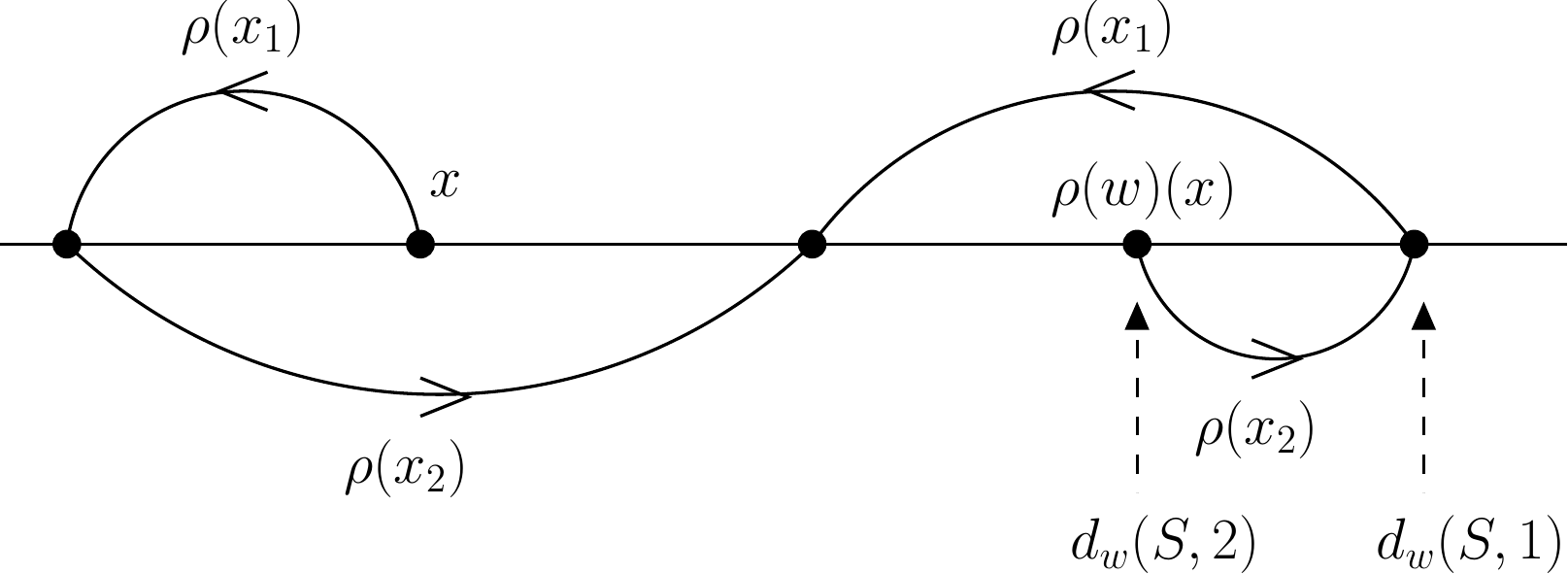}
\caption{This shows a possible example of a trajectory $S=S(\rho,w,x)$ for $w=x_2^{-1}x_1^{-1}x_2x_1$. In this depiction we can see the sets $D_w(S,i)$ as the starting points of the arrows marked with $\rho(x_i)$.}
\label{f1}
\end{figure}

We find it enlightening to think of a trajectory as a graph, as shown in Figure \ref{f1}. This can also be formulated for general sequences. With the notations of Definition \ref{trajectory}, we associate a {\em trajectory graph} to $S=(s_0,\ldots,s_m)$ as follows: The vertex set is just $\{s_0,\ldots,s_m\}$, and we put in an edge for every $j=1,\dots,m$ connecting $s_{j-1}$ to $s_j$. This edge is labeled and oriented according to $a_j$. Namely, if $a_j=x_{i}^{\epsilon}$ then the edge label is $x_i$, and its orientation depends on $\epsilon$: when $\epsilon=1$ it is oriented from $s_{j-1}$ towards $s_j$, and the reverse for $\epsilon=-1$.

Looking at the trajectory graph of $S$, we can regard $D_w(S,i)$ as the set of vertices which have an outgoing edge marked with $x_i$.   

The following observation will be useful in the proof of most of the heavier results. It comes up naturally by looking at trajectory graphs.

\begin{rem}\label{inverso} For any $\rho\in Rep(\F_n,Homeo_{+}(\R))$, $w\in\F_n$ and $p\in\R$ we have that $S(\rho,w^{-1},\rho(w)(p))$ is just $S(\rho,w,p)$ traversed backwards. Thus $$D_w(S(\rho,w,p),i)=D_{w^{-1}}(S(\rho,w^{-1},\rho(w)(p)),i)$$ for all $i=1,\ldots,n$.
\end{rem}
 
Our first result is the surjectivity of the word map. Its proof will illustrate the technique of defining a representation by pre-fixing some trajectories.

\begin{lem}\label{tronco2} Let $w\in\F_n$ and $g\in Homeo_{+}(\R)$. Then  $$V_{w}(g) = \{\rho\in\text{Rep}(\F_n,Homeo_{+}(\R)): \rho(w)=g \}$$ is non empty.
\end{lem}

\begin{proof}

We will prove it for the case $Fix(g)=\emptyset$, the general case reduces to this by setting $Fix(g)$ as global fixed points of $\rho$. We can further assume that $g(x)>x$ for every $x\in\R$, the other case being analogous.

 It suffices to find a representation $\rho_0$ such that $\rho_0(w)(x)>x$ for all $x\in\R$, since then we can conjugate $\rho_0$ to obtain a representation $\rho$ with $\rho(w)=g$. The same reasoning allows us to exchange $w$ by one of its conjugates in $\F_n$, so we may assume that $w$ is cyclically reduced.

Let $m=|w|$ and set $S_r = (mr,mr+1,...,m(r+1))$ for all $r\in \Z$. In other words, $S_r=(s_{r,0},\ldots,s_{r,m})$ is the sequence defined by $s_{r,j}=mr+j$. Notice that $S_r$ ends where $S_{r+1}$ begins. We will give a representation $\rho_0$ that has all $S_r$ as trajectories by $w$. This is enough, since then we shall have that $\rho_0(w^r)(0)=mr \to \pm\infty$ as $n\to\pm\infty$, proving that $\rho_0(w)(x)>x$ for all $x\in\R$.

Fix some $r\in\Z$. For each $i\in\{1,...,n\}$, we define a map $g_i$ on $D_w(S_r,i)$ by the formula $g_i(s_{r,j-1})=s_{r,j}$ if $a_j=x_i$ and $g_i(s_{r,j})=s_{r,j-1}$ if $a_j=x_i^{-1}$. These are indeed well defined maps, and are also injective, due to the fact that $S_r$ has no repetitions and $w$ is reduced (admits no cancellations). This can be seen by considering the trajectory graph for $S_r$: The formula for $g_i$ is just defined by the arrows (oriented edges) marked by $x_i$, and we notice that no vertex admits two incoming or two outgoing edges of the same label. 

Assuming that $w$ is cyclically reduced allows us to join, for fixed $i\in\{1,...,n\}$, all the maps $g_i$ corresponding to every $r\in\Z$ (we abuse notation by omitting $r$ from it). This gives rise to a well defined map $g_i$ on $\bigcup_{r\in\Z}D_w(S_r,i)$, as can be shown by the same argument with the trajectory graphs, where the fact that $w$ is cyclically reduced is used for the common vertices of different sequences, namely $s_{r,m}=s_{r+1,0}$. 

We will show that $g_i:\bigcup_{r\in\Z}D_w(S_r,i)\to\R$ are increasing (i.e. $g_i(x)<g_i(y)$ for $x<y$), and therefore can be extended to $\R$ as homeomorphisms (e.g. linear interpolation). This defines the maps $\rho_0(x_i)$ that give the desired representation $\rho_0$. Notice that if a generator $x_i$ is not present in $w$, we can choose $\rho_0(x_i)$ freely.

Now we see that $g_i$ is increasing. Take $x<y$ in $X_i = \bigcup_{r\in\Z}D_w(S_r,i)$. The construction of the sequences $S_r$ gives that $g_i(x)=x+\epsilon(x)$ for $x\in X_i$ and $\epsilon(x)=\pm 1$. This shows $g_i(x)<g_i(y)$ directly for $y-x\geq 3$, and using injectivity for $y-x=2$. Also notice that the union of the trajectory graphs for $S_r$ admits no closed edge-paths. So if both $x$ and $x+1$ belong in $X_i$ then we cannot have $\epsilon(x)>\epsilon(x+1)$, which gives $g_i(x)<g_i(y)$ for $y-x=1$.




\end{proof}











The following is the main technical lemma. It refines Lemma 3.1 in \cite{abr}. As we mentioned earlier, its aim is to give a perturbation of a representation of $\F_n$ in $Homeo_{+}(\R)$ that achieves many desired properties: having non trivial stabilizer in the orbit of a given point, preserving some given trajectories of the original action, and controlling the dynamics near $+\infty$ of a specified element $w\in\F_n$.

\begin{lem}\label{tronco} Let $\rho\in Rep(\F_n,Homeo_{+}(\R))$, $p,p_1,\ldots,p_k\in\R$ and $w,v_1,\ldots,v_k\in\F_n$, with $w$ cyclically reduced. Let $$d_{ij} = d_{v_j}(S(\rho,v_j,p_j),i) \text{ for } i=1,\ldots,n \text{ and } j=1,\ldots,k$$ and $$d_i=d_w(S(\rho,w,p),i) \text{ for }  i=1,\ldots,n$$
Assume that $\rho(w)(p)\neq p$, and that there exists $i_0$ with $d_{i_0}\geq \max_j d_{i_0j}$.

\vs
 
 Then there exists $\rho'\in Rep(\F_n,Homeo_{+}(\R))$ such that:
\begin{enumerate}
\item $\rho'(x_i)$ agrees with $\rho(x_i)$ on $(-\infty,m_i]$ where $m_i=\max\{d_i,d_{i1},\ldots,d_{ik}\}$, for each $i=1,\ldots,n$.
\item $S(\rho',v_j,p_j)=S(\rho,v_j,p_j)$ for every $j=1,\ldots,k$.
\item $\rho'(w)$ agrees with $\rho(w)$ on $(-\infty,p]$ and $\rho'(w)(x)\neq x$ for $x\geq p$.
\item $p$ has non trivial stabilizer under $\rho'$.
\end{enumerate}
\end{lem} 

\begin{rem} In the statement of Lemma \ref{tronco}, point 2 and the first part of point 3 can be deduced from point 1. This is straightforward from the definitions. Nevertheless, we state these consequences explicitly as they are important in the applications.
\end{rem}

The hypothesis of Lemma \ref{tronco} is a bit involved. The following gives an easier way to verify it, that will suffice in most cases.

\begin{lem}\label{increasing} With the notations of Lemma \ref{tronco}, if $p\geq m = \max\{S(\rho,v_j,p_j):j=1,\ldots,k\}$ then there exists $i_0$ with $d_{i_0}\geq \max_j d_{i_0j}$.
\end{lem}
\begin{proof} Write $w=a_m\cdots a_1$ and $a_1 = x_{i_0}^{\epsilon}$. We claim this $i_0$ works. 

According to the sign of $\epsilon$, either $p$ or $\rho(a_1)(p)$ belongs to $D_w(S(\rho,w,p),i_0)$. If $p$ does (when $\epsilon =1$), then $d_{i_0}\geq p \geq m \geq  d_{i_0j}$ for every $j$. 

Otherwise, $\epsilon =-1$ and $\rho(a_1)(p)\leq d_{i_0}$. If the assertion were not true, then $\rho(a_1)(p)\leq d_{i_0}<  d_{i_0j}\leq m \leq p$ for some $j$. But then we should have that $\rho(x_{i_0})(d_{i_0j})> p$. This is a contradiction since $\rho(x_{i_0})( d_{i_0j})$ belongs to $S(\rho,v_j,p_j)$, so it should be less than $m$. 

\end{proof}

The proof of Lemma \ref{tronco} gets very technical. We point out that Theorems \ref{manija}, \ref{hnn} and \ref{positiveword} can be derived from the statement of \ref{tronco} without using facts that come up in its proof. 
\vs

{\bf \em Proof of Lemma \ref{tronco}:} For simplicity, we will assume that $\rho(w)(p)>p$. 
Otherwise we exchange $w$ for $w^{-1}$ and $p$ for $\rho(w)(p)$. This can be done without altering the hypotheses by Remark \ref{inverso}. 

The strategy of the proof follows the same idea we used for Lemma \ref{tronco2}. We will first define some suitable sequences $S_r$ for $r=1,2,...$ that will become the trajectories $S(\rho',w,\rho'(w^r)(p))$ for the desired representation $\rho'$. Those sequences will be used to define the maps $\rho'(x_i)$ in some discrete sets, and then it will be possible to extend them by interpolation. Choosing the $S_r$ carefully will allow us to achieve the objectives in the statement: make these extensions possible with $\rho'(x_i)$ and $\rho(x_i)$ agreeing on $(-\infty,m_i]$,  make $\rho'(w^r)(p)$ tend to $+\infty$ (thus $\rho'(w)$ will have no fixed points after $p$), and also make $\rho'(z)(p)=p$ for some non trivial $z\in\F_n$ (non trivial stabilizer).

Write $w=a_m\cdots a_1$, with $m=|w|$. Recall from the statement that $m_i=\max\{d_i,d_{i1},\ldots,d_{ik}\}$, and let $M=\max S(\rho,w,p)\cup\bigcup_{j=1}^k S(\rho,v_j,p_j)$.

\vs

{\flushleft \bf Step 1:} Definition of $S_1$.
\vs

We define $S_1 = (s_{1,0},\ldots,s_{1,m})$ as follows:

Set $s_{1,0}=\rho(w)(p)$, and for $0<j\leq m$, let $$ s_{1,j}= \rho(a_j)(s_{1,j-1}) \text{ while } \left\{ \begin{array}{ccc} s_{1,j-1} \leq m_i & \text{ if } & a_j=x_i, \text{ or } \\ s_{1,j-1} \leq \rho(x_i)(m_i) & \text{ if } & a_j=x_i^{-1}. \end{array} \right. $$ We get to $s_{1,l}$, the last element we can define by that process. Then we choose $s_{1,l+1}>M$, and set $s_{1,j+1}=s_{1,j}+1$ for every $j= l+1,\ldots,m$.

The above definition of the $s_{1,j}$ amounts to say that $S_1$ agrees with $S(\rho,w,\rho(w)(p))$ for as long as $s_{1,j}$ can be defined using the $\rho(x_i)$ restricted to $(-\infty,m_i]$. When that is no longer possible, we have freedom to pick the next $s_{1,j}$ without contradicting point 1 in the statement. We will do so to help us achieve point 3, in a similar fashion as in Lemma \ref{tronco2}. Figure \ref{f2} gives an example of this process. The next claim says that the last part of this definition really takes place.

\vs

{\bf Claim:} $l<m$.  
\vs

\begin{proof} Recall the hypothesis that there is some $i_0$ with $d_{i_0}\geq \max_j d_{i_0j}$. Also recall this means $d_{i_0}=m_{i_0}$. Notice that $d_w(S(\rho,w,x),i)$ is increasing in $x$, as it is a maximum of increasing homeomorphisms. Since we are assuming $\rho(w)(p)>p$, this gives us that $$ d_w(S(\rho,w,\rho(w)(p)),i_0) > d_w(S(\rho,w,p),i_0) = d_{i_0} = m_{i_0}$$ 
On the other hand, if $l=m$ we would have that $S_1=S(\rho,w,\rho(w)(p))$ and every point in $D_w(S(\rho,w,\rho(w)(p)),i)$ would be less than $m_i$ for every $i$, for that is what is needed for the ``while'' condition to hold through $j=1,\ldots,m$. This contradicts what we just obtained for $i_0$.

\end{proof}

\begin{figure}[h!] 
\includegraphics[width=1\textwidth]{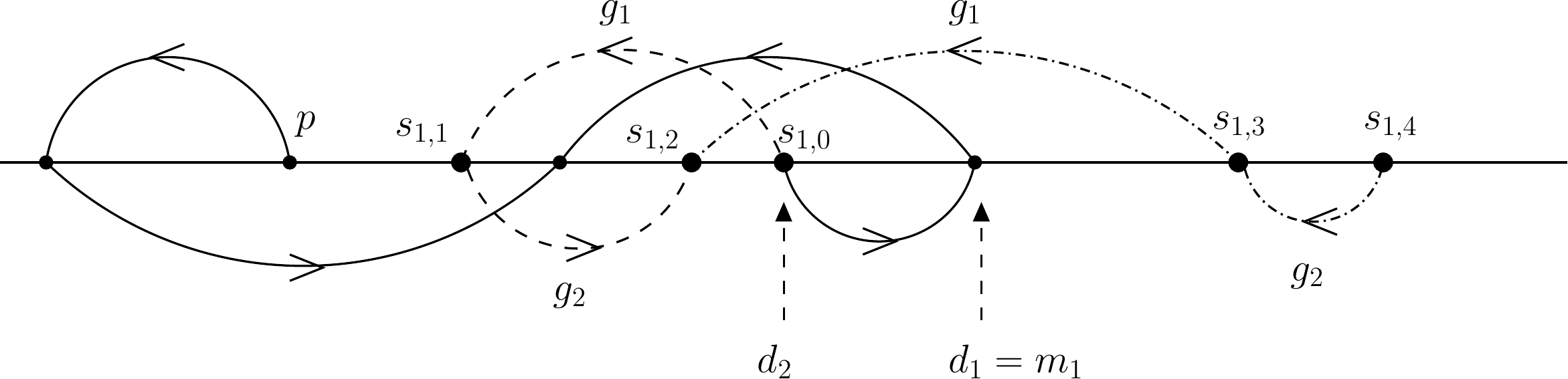}
\caption{Step 1: we show the construction of the sequence $S_1$ for the example in Figure \ref{f1}, assuming that $d_1=m_1$ (the sequences $S(\rho,v_j,p_j)$ are not drawn in the picture). Here $l=2$. Step 3: we draw the arrows defining the $g_i$ on $D_w(S_1,i)$.}
\label{f2}
\end{figure}

\vs
{\flushleft \bf Step 2:} Definition of $S_r$ for $r\geq 2$.
\vs


The idea is to make a suitable $S_2$ that will help us create a non trivial stabilizer, and define $S_r$ for $r\geq 3$ in a similar way as in Lemma \ref{tronco2}, in order to obtain point 3 in the statement.  We remark that $S_r$ must begin where $S_{r-1}$ ends.
We need to distinguish two cases according to the form of $w$. They will only differ in $S_2$ and possibly in $S_3$.

\begin{itemize}
\item[A:] Suppose $n=2$ and $w=[x_1,x_2]^u$, where $[x_1,x_2]=x_2^{-1}x_1^{-1}x_2x_1$ and $m=4u$. This case applies for the other commutators of $x_1$ and $x_2$ as well, possibly exchanging the generators, or replacing them by their inverses. 

We define $S_2=(s_{2,0},\ldots,s_{2,4u})$ as  $s_{2,0}=s_{1,m}$, $s_{2,1}=s_{1,m}+3$, $s_{2,2}=s_{1,m}+4$, $s_{2,3}=s_{1,m}+1$, $s_{2,4}=s_{1,m}+2$, and if $u>1$ we set $s_{2,5}=s_{1,m}+5$ and $s_{2,j+1}=s_{2,j}+1$ for $j\in \{5,\ldots,4u-1\}$. 

For $r\geq 3$ we define $S_r=(s_{r,0},\ldots,s_{r,4u})$ by $s_{r,0}=s_{r-1,4u}$ and $s_{r,j+1}=s_{r,j}+1$ for $j\in\{0,\ldots,4u-1\}$, with the exception of $s_{3,1}=s_{3,0}+3$ if $u=1$. (Notice that by this exception, the trajectory graph defined on $\bigcup_{r\geq 2}S_r$ looks the same for any $u$, and avoids two different incoming edges at $s_{2,1}$ both marked with $x_1$). 
\item[B:] For any other $w$ we set $S_r= (s_{r,0},\ldots,s_{r,m})$ with $s_{r,0}=s_{r-1,m}$ and $s_{r,j+1} = s_{r,j}+1$ for every $r\geq 2$ and $j\in\{0,\ldots,m-1\}$. 
\end{itemize}

\begin{figure}[h!] 
\centering
\includegraphics[width=0.8\textwidth]{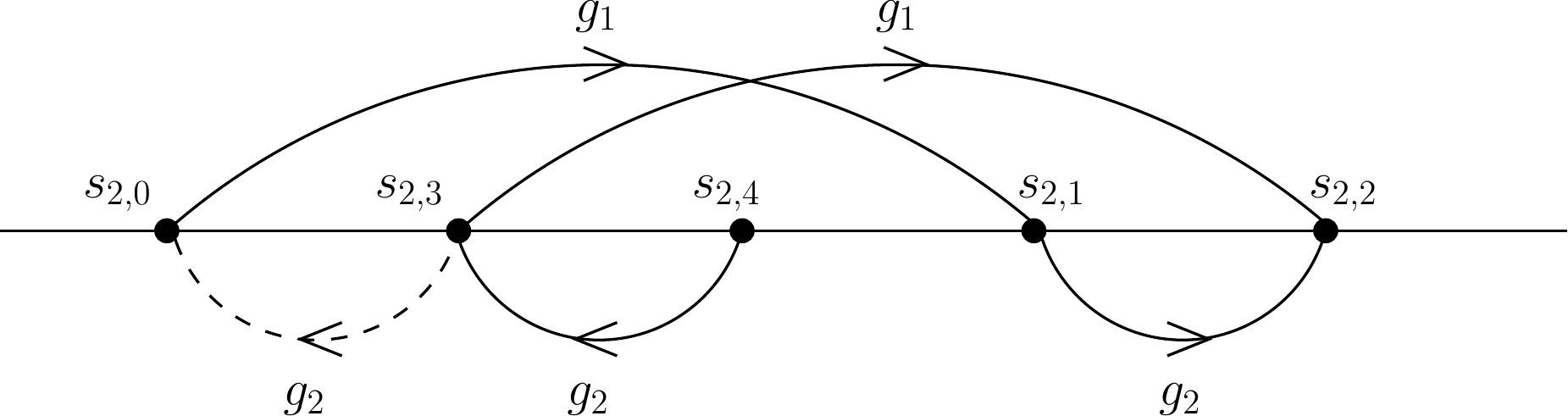}
\caption{Case A. Step 2: we draw the first four points of the sequence $S_2$. Step 3: we draw the arrows corresponding to the maps $g_i$.}
\label{f3}
\end{figure}

\vs
{\flushleft \bf Step 3:} Partial definition of the maps $\rho'(x_i)$.
\vs

The sequences $S_r$ define maps $g_i$ on the sets $D_w(S_r,i)$ as in Lemma \ref{tronco2}, by taking $g_i(s_{r,j-1})=s_{r,j}$ if $a_j=x_i$ and $g_i(s_{r,j})=s_{r,j-1}$ if $a_j=x_i^{-1}$. We can set $g_i$ to agree with $\rho(x_i)$ on $(-\infty,m_i]$ and we obtain well defined maps $g_i:(-\infty,m_i]\cup\bigcup_r D_w(S_r,i)\to\R$, that are also injective. (This works by the same arguments used in Lemma \ref{tronco2}. Here we use that $w$ is cyclically reduced).

In order to achieve the non trivial stabilizer, we make a further extension of some of the $g_i$, that will be different on each case. 
  
\begin{itemize}
\item In Case A we set $g_2(s_{2,3})=s_{2,0}$. (See Figure \ref{f3}). Notice there is no arrow coming from $s_{2,3}$ and marked by $x_2$ in the trajectory graphs, thus $g_2$ is well defined. Injectivity is given by the fact that there is no incoming edge marked by $x_2$ at $s_{1,m}=s_{2,0}$, by construction of $S_1$ and the form of $w=[x_1,x_2]^u$. (There is an outgoing edge marked by $x_2$ at $s_{2,0}$, but that is compatible).
\item In Case B there is some $j_1\in\{1,\ldots,m\}$ and some $i_1\in\{1,\ldots,n\}$ such that $a_{j_1+1} \neq x_{i_1}^{\pm 1} $, and we do not have both $a_{j_1}=x_{i_1}^{\epsilon}$ and $a_{j_1+2}= x_{i_1}^{-\epsilon}$, where the indexes are taken mod $m=|w|$. Otherwise we would be in case A.   

If neither $a_{j_1}$ nor $a_{j_1+2}$ is $x_{i_1}^{\pm 1}$, we can set $g_{i_1}(s_{2,j_1})=s_{2,j_1+1}$ without contradicting the values of $g_{i_1}$ we had defined previously. This is also true if $a_{j_1} = x_{i_1}$ or $a_{j_1+2}= x_{i_1}$.   
If $a_{j_1}=x_{i_1}^{-1}$ or $a_{j_1+2}=x_{i_1}^{-1}$, then we cannot do that, but we can put $g_{i_1}(s_{2,j_1+1})=s_{2,j_1}$. (See Figure \ref{f4}).

\end{itemize}

\begin{figure}[h!] 
\centering
\includegraphics[width=0.8\textwidth]{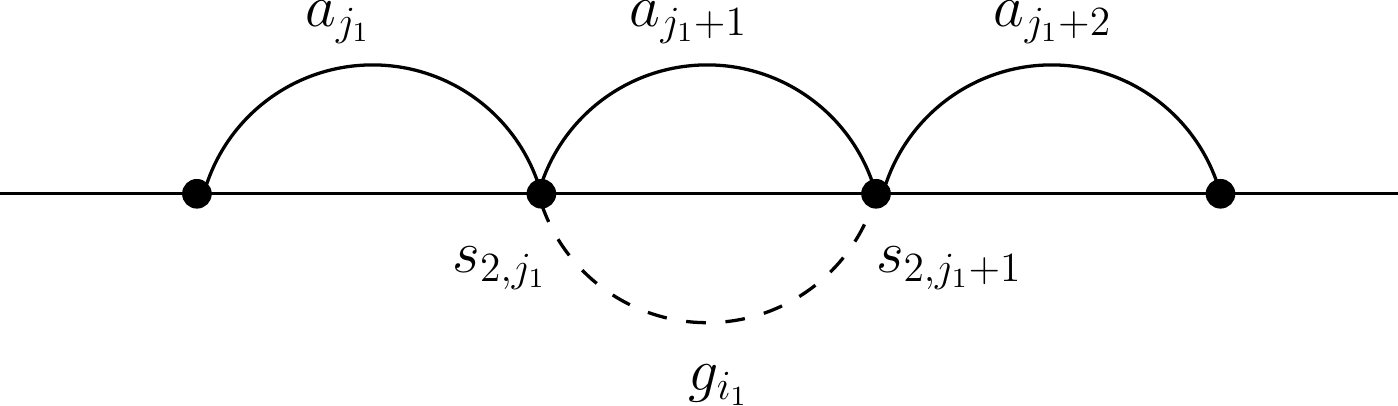}
\caption{Case B. We have to pick an orientation for the dotted line, ensuring that there is no vertex with two incoming or outgoing arrows marked with $g_{i_1}$.}
\label{f4}
\end{figure}

\vs
{\flushleft \bf Step 4:} Extension of the maps $\rho'(x_i)$.
\vs

From the previous step we have maps $g_i$ defined on $(-\infty,m_i]\cup X_i$ for $X_i$ a discrete set. As we did for Lemma \ref{tronco2}, we will check that these maps are increasing, namely that $x<y$ implies $g_i(x)<g_i(y)$, so they can be extended to $\R$ as homeomorphisms. This completes the definition of the maps $\rho'(x_i)$ that give the representation $\rho'$.

\vs
{\bf Claim:} $g_i:(-\infty,m_i]\cup X_i\to\R$ is increasing.
\vs

\begin{proof} On the sets $X_i$ we can argue as in Lemma \ref{tronco2}, for case A it may be helpful to look at Figure \ref{f3}. It remains to show that $g_i(x)<g_i(y)$ when $x=m_i$ and $y=\min X_i$. The proof splits into two cases depending on $i$: Recall the construction of $S_1$, and let $i_*$ be such that $a_l=x_{i_*}^{\epsilon}$. 

If $i\neq i_*$, then $g_i$ maps $X_i-(-\infty,m_i]$ into $(M,+\infty)$, since we picked $s_{r,j}>M$ for $r>1$ and $r=1,j>l$. By definition $M\geq m_i$ for all $i$, hence $g_i$ is increasing. 

For $i_*$ we discuss according to $\epsilon=\pm 1$. If $\epsilon=1$, then by our definition of $l$ we have $s_{1,l}>m_{i_*}$, and we had set $g_{i_*}(s_{1,l})=s_{1,l+1}>M>\rho(x_{i_*})(m_{i_*})=g_{i_*}(m_i)$. We also have $y=s_{1,l}$, so this gives the claim. In case $\epsilon=-1$, we have $s_{1,l}>\rho(x_{i_*})(m_{i_*})$ (also by definition of $l$) and we had set $g_{i_*}(s_{1,l+1})=s_{1,l}>\rho(x_{i_*})(m_{i_*})=g_{i_*}(m_{i_*})$. Notice that $y=s_{1,l+1}$ (definition of $l$, and picking $s_{1,l+1}>M\geq m_{i_*}$), so this case is finished as well.

\end{proof}

\vs
{\flushleft \bf Step 5:} Verification of the properties of $\rho'$.
\vs

Point 1 in the statement is clear, and implies point 2 and the first part of point 3. Notice that $S_r$ is the trajectory of $\rho'(w^r)(p)$ by $w$ under the representation $\rho'$. So we have that $\rho'(w^r)(p)=s_{r,0}$, which tends to $+\infty$ with $r$, and thus $\rho'(w)$ has no fixed points after $p$.

For point 4 notice that $s_{r,j}$ is in the orbit of $p$ for all $r\geq 1$ and $j=0,\ldots,m$. Then we have:

\begin{itemize}
\item In Case A, $s_{2,0}$ is fixed by $\rho'(x_2x_1^{-1}x_2x_1)$.
\item In Case B, $s_{2,j_1}$ is fixed by $\rho'(x_{i_1}^{\epsilon}a_{j_1+1})$ for some $\epsilon=\pm 1$.
\end{itemize} 

This gives that the orbit of $p$ has non trivial stabilizer under $\rho'$.

$\hfill \Box$

\section{Proof of Theorem \ref{manija}}

Let $\F_n = \langle x_1,...,x_n\rangle$. Then a group $G = \F_n*_{\Z}H$ as in the statement can be written as $\F_n*_{w=h}H$ where $w\in\F_n$ is a cyclically reduced word and $h\in H$.   

Let $<$ be an order on $G$ and $\rho$ a dynamical realization for $<$ (based at 0). We shall define a new representation $\bar\rho$ that is a small perturbation of $\rho$ and has non trivial stabilizer on the orbit of 0. Then the theorem will follow from Proposition \ref{aislado}. 

Let $g=\rho(h)=\rho(w)$. Let $V_{w}(g) = \{\rho'\in\text{Rep}(\F_n,Homeo_{+}(\R)): \rho'(w)=g \}$. We will find some $\rho'\in V_{w}(g)$ close to $\rho|_{\F_n}$ so that the representation $\bar\rho$ defined on $G$ by gluing it with $\rho|_H$ has non trivial stabilizer on the orbit of 0. 



Take an arbitrarily large compact interval $K\subseteq \R$. Since $\rho$ has no global fixed points, being a dynamical realization \cite{GOD}, we can build a finite $\rho$-trajectory $S$, with respect to the generating system $\{H,x_1,\ldots,x_n\}$, which is an increasing sequence begining at $0$ and ending outside $K$. By a $\rho$-trajectory with respect to $\{H,x_1,\ldots,x_n\}$ we mean that each point of $S$ is obtained from the previous one by acting either by $\rho(h)$ for $h\in H$ or by $\rho(x_i)^{\pm 1}$ for some $i\in\{1,\ldots,n\}$.  

By keeping track of the generators, we can split $S$ into a union of trajectories by $\rho|_{\F_n}$ that are connected by the action of elements of $H$ (where the union of trajectories is sequence concatenation). Namely, we can write $S=\bigcup_{j=1}^k S(\rho|_{\F_n},v_j,p_j)$ where $v_j\in\F_n$, and you obtain $p_{j+1}$ by acting on $\rho(v_j)(p_j)$  by some $\rho(h_j)$ for $h_j\in H$. Notice that $p_1=0$. Observe also that if $\rho|_H$ has no global fixed point in $K\cap[0,+\infty)$ we may choose $S=(0,\rho(h)(0))$, the trajectories by $\rho|_{\F_n}$ are just points, and the $v_j$ are trivial.

Let $p=\rho(v_k)(p_k)$ be the last element of $S$. 



\vs
{\flushleft \bf Claim:} The representation $\rho|_{\F_n}$, and the $p,p_1,\ldots,p_k\in\R$ and $w,v_1,\ldots,v_k\in\F_n$ we just constructed, are in the hypothesis of Lemma \ref{tronco}.
\vs

\begin{proof}
Recall we are assuming $w$ cyclically reduced. That $\rho(w)(p)\neq p$ is clear since $\rho$ is the dynamical realization of a total order and $p$ is in the orbit of $0$. Then we can apply Lemma \ref{increasing} since $p=\max S$.
\end{proof}

Let $\rho_0$ be the representation obtained from $\rho|_{\F_n}$ applying Lemma \ref{tronco}. Let $q=\min Fix(g)\cap (p,+\infty)$, with the convention that $\min\emptyset=+\infty$. The case for $q=+\infty$ is simpler, so we will focus on the construction when $q<+\infty$. We conjugate $\rho_0$ to an action on $(-\infty,q]$ with $q$ as a global fixed point, by a conjugation $\psi:\R\to (-\infty,q)$ that restricts to the identity up to $\max\{p,\rho(w)(p)\}$ and such that $\rho(w)|_{(-\infty,q]} = \psi \rho_0(w) \psi^{-1}$. This is possible by Lemma \ref{promotion}.

We use Lemma \ref{tronco2} to get a representation $\rho_1$ of $\F_n$ on $(q,+\infty)$ with $\rho_1(w)=g|_{(q,+\infty)}$. We define $\rho'\in \text{Rep}(\F_n,Homeo_{+}(\R))$ by $$\rho'(\gamma)(x)=\left\{ \begin{array}{ccc} \psi\rho_0(\gamma)\psi^{-1}(x) & \text{if}& x\leq q \\ \rho_1(\gamma)(x) &\text{if}& x>q \end{array}\right.$$  

Notice that $\rho'(w)=g$, so we get $\bar{\rho}\in \text{Rep}(G,Homeo_{+}(\R))$ that agrees with $\rho'$ on $\F_n$ and with $\rho|_H$ on $H$. Choosing $K$ large enough we can get $\bar{\rho}$ arbitrarily close to $\rho$. It remains to show that the orbit of $0$ has non trivial stabilizer. By Lemma \ref{tronco} we have $S(\rho_0,v_j,p_j)=S(\rho,v_j,p_j)$ for all $j=1,\ldots,k$, and since $\psi$ is the identity up to $\max\{p,\rho(w)(p)\}$ these trajectories are not affected by the conjugation, so $S(\rho',v_j,p_j)= S(\rho,v_j,p_j)$. Thus $p$ is in the $\bar{\rho}$ orbit of $0$. On the other hand, the stabilizer of $p$ by $\rho_0$ is non trivial by Lemma \ref{tronco}, and $\psi(p)=p$, so $p$ has the same stabilizer under $\rho'$. The stabilizer of $p$ under $\bar{\rho}$ contains it, so it is non trivial. This gives the theorem by Proposition \ref{aislado}.

\section{Proof of Theorem \ref{hnn}} 

Let $G=\langle t,x_1,...,x_n|tw_1t^{-1}=w_2\rangle$ as in the statement. We can assume that $w_1$ and $w_2$ are cyclically reduced (possibly taking an equivalent presentation). 

Let $<$ be an order on $G$ and $\rho$ its dynamical realization. The proof follows the same strategy as that of Theorem \ref{manija}: we shall define a new representation $\bar\rho$ that is a small perturbation of $\rho$ and has non trivial stabilizer on the orbit of 0. Then finish by Proposition \ref{aislado}.

Take an arbitrarily large compact interval $K\subseteq \R$. As we did for Theorem \ref{manija}, we take a finite increasing $\rho$-trajectory $S$, with respect to the generating set $\{t,x_1,...,x_n\}$, that begins at $0$ and ends outside $K$. This means each point of $S$ is obtained from the previous one by acting either by $\rho(t)^{\pm 1}$ or by $\rho(x_i)^{\pm 1}$ for some $i\in\{1,\ldots,n\}$. 

Again as in Theorem \ref{manija}, we split $S$ into a union of trajectories by $\rho|_{\F_n}$. So we write $S=\bigcup_{j=1}^k S(\rho|_{\F_n},v_j,p_j)$, where $p_1=0$ and $p_{j+1}$ is obtained by acting on $\rho(v_j)(p_j)$ by $\rho(t)^{\pm 1}$. (It is possible that $v_j$ may be trivial, in which case the trajectory $S(\rho|_{\F_n},v_j,p_j)$ is the single point $p_j$). 

Let $p=\rho(v_k)(p_k)$ be the last element of $S$.

Consider the word $\bar w=t^{-1}w_2^{-1}tw_1$ that represents the identity on $G$. Notice that the trajectory $S(\rho,\bar w,p)$ is decomposed as follows: the initial segment is $S(\rho,w_1,p)$, next comes $S(\rho,w_2,\rho(t)(p))$ traversed backwards, and then the final point is $p$ (since $\bar w$ is the identity in $G$, the trajectory must be closed).  

As in the statement of Lemma \ref{tronco}, we consider $$d_{ij} = d_{v_j}(S(\rho,v_j,p_j),i) \text{ for } i=1,\ldots,n \text{ and } j=1,\ldots,k$$
Define as well $d_i^{(1)} = d_{w_1}(S(\rho,w_1,p),i)$ and $d_i^{(2)}=d_{w_2}(S(\rho,w_2,\rho(t)(p)),i)$.   

By Lemma \ref{increasing} there exists $i_0$ with $d_{i_0}^{(1)}\geq\max_j d_{i_0j}$. (Recalling that $S$ is an increasing sequence and $p$ its last point). We will assume that $d_{i_0}^{(1)}\geq d_{i_0}^{(2)}$ for simplicity. Otherwise, we can invert the roles of $w_1$ and $w_2$, by replacing $S$ by $S\cup\{\rho(t)(p)\}$ and $\bar w$ by $tw_1^{-1}t^{-1}w_2$. This does not affect the $d_{i}^{(1)}$ and $d_{i}^{(2)}$ by Remark \ref{inverso}.

We will apply Lemma \ref{tronco} to $\rho|_{\F_n}$, with the following setting:
\begin{itemize}
\item $v_1,\ldots,v_k$ as constructed above, $v_{k+1}=w_2$, and $w=w_1$.
\item $p_1,\ldots,p_k$ and $p$ as constructed above, and $p_{k+1}=\rho(t)(p)$.    
\end{itemize}

Notice the hypotheses of Lemma \ref{tronco} hold for this setting, where the notation in the statement would be $d_{i,k+1}= d_i^{(2)}$ and $d_i=d_i^{(1)}$. From Lemma \ref{tronco} we obtain a representation $\rho_0\in \text{Rep}(\F_n,Homeo_{+}(\R))$.


 






This $\rho_0$ is a perturbation of $\rho|_{\F_n}$ that satisfies what we would expect from the restriction to $\F_n$ of our desired representation. However, it is not always possible to extend it to $G$. We deal with this problem in what follows, splitting the procedure into two cases. For both we will need to notice that $\rho_0(w_1)$ agrees with $\rho(w_1)$ up to $p$, and $\rho_0(w_2)$ agrees with $\rho(w_2)$ up to $\rho(t)(p)$. These maps are partially conjugated by $\rho(t)$, meaning that for $x\in(-\infty,p]$ we have $\rho(t)\rho_0(w_1)(x)=\rho_0(w_2)\rho(t)(x)$.

\begin{itemize}


\item[Case 1] If $\rho_0(w_2)$ has no fixed points greater than $\rho(t)(p)$, then it is weakly conjugated to $\rho_0(w_1)$ by a map that is strong on $(-\infty,p]$ (see definition \ref{weakdefi}). Thus by Lemma \ref{promotion} we can find a homeomorphism $\varphi$ that conjugates $\rho_0(w_1)$ and $\rho_0(w_2)$, and agrees with $\rho(t)$ on $(-\infty,p]$. In this case we can define $\bar \rho(t)=\varphi$ and $\bar \rho|_{\F_n}=\rho_0$. 

\item[Case 2] On the other hand, assume that $\rho_0(w_2)$ has fixed points greater than $\rho(t)(p)$. Let $M=\max S\cup S(\rho,\bar w,p)$. As in the proof of Theorem \ref{manija}, we can conjugate $\rho_0$ to a representation on $(-\infty,q]$ for some $q>M$, that has $q$ as a global fixed point, by a map that is the identity on $(-\infty,M]$. Call the result of this conjugation by $\rho'_0$.

Let $z$ be the first fixed point of $\rho'_0(w_2)$ on $[\rho(t)(p),q]$. Then $z<q$ by the assumption of this case. Notice that $\rho'_0(w_1)$ and $\rho'_0(w_2)|_{(-\infty,z]}$ are weakly conjugated, because of the definition of $z$ and their partial conjugation up to $p$ by $\rho(t)$. By Lemma \ref{promotion} we can find $\varphi:(-\infty,q]\to(-\infty,z]$ a homeomorphism that conjugates $\rho'_0(w_1)$ to $\rho'_0(w_2)|_{(-\infty,z]}$ and agrees with $\rho(t)$ on $(-\infty,p]$.

Extend $\varphi$ to a homeomorphism of $\R$ that takes $[q,q+1]$ to $[z,q]$ and $[q+m,q+m+1]$ to $[q+m-1,q+m]$ for $m\geq 1$.

Applying Lemma \ref{tronco2} recursively, we can define $$\rho'_m\in Rep(\F_n,Homeo_{+}([q+m-1,q+m]))$$ so that $\rho'_m(w_1)=\varphi^{-1}\rho'_{m-1}(w_2)\varphi$ for $m\geq 1$, where the base case is $\rho'_{1}(w_1)=\varphi^{-1}\rho'_0(w_2)|_{[z,q]}\varphi$.

We define $\rho'\in\text{Rep}(\F_n,Homeo_{+}(\R))$ so that $\rho'(\gamma)|_{(-\infty,q]} = \rho'_0(\gamma)$ and \\ $\rho'(\gamma)|_{(q+m-1,q+m]}=\rho_m(\gamma)$ for $m\geq 1$. Then we define $\bar\rho\in\text{Rep}(G,Homeo_{+}(\R))$ by $\bar\rho|_{\F_n}=\rho'$ and $\bar\rho(t)=\varphi$.

\end{itemize}

In both cases the resulting $\bar\rho$ is a perturbation of $\rho$ that can be made arbitrarily small by choosing $K$ large enough. The stabilizer of $p$ by $\rho_0$ is non trivial and this fact is preserved by the modifications in case 2. On the other hand, the trajectory $S$ is not affected either, since $S(\rho,v_j,p_j)=S(\bar\rho,v_j,p_j)$ for all $j$, and $\bar\rho(t)$ agrees with $\rho(t)$ on $(-\infty,p]$. Thus the stabilizer of $0$ by $\bar\rho$ is non trivial. The theorem follows from Proposition \ref{aislado}.

\section{Proof of Theorem \ref{positiveword}}

Throughout this section we consider $\bar w=tw_1tw_2tw_3$ as a word that represents the identity in $G$.

\begin{lem} \label{lasmagias} $G$ can be presented as $G=\langle t,x_1,...,x_n|tw_1tw_2tw_3=1\rangle$ where the strings $w_iw_j^{-1}$ are cyclically reduced for $i\neq j\in\{1,2,3\}$.
\end{lem}

\vs
\begin{proof} Among the presentations of $G$ of the form $\langle t,x_1,...,x_n|tw_1tw_2tw_3=1\rangle$ we choose one so that $|w_1|+|w_2|+|w_3|$ is minimal. We will show that this presentation is the desired one. 
The string $w_iw_j^{-1}$ is cyclically reduced when $w_i$ and $w_j$ do not have a common initial segment nor a common final segment. If this is not the case, we will give another presentation that contradicts the minimality assumption. Suppose for instance that $w_1=av_1$ and $w_2=av_2$ where $a$ is a common initial segment. Then we change the presentation by the Tietze transformation $t'=ta$, which gives $\langle t',x_1,...,x_n|t'w'_1t'w'_2t'w'_3=1\rangle$ where $w'_1=v_1$, $w'_2=v_2$, and $w'_3=a^{-1}w_3$ (after a possible reduction). Then $|w'_1|+|w'_2|+|w'_3| \leq |w_1|+|w_2|+|w_3| - |a|$ contradicting the minimality. For a final segment, or other values of $i,j$, the argument is very similar.
\end{proof}

\begin{defi} Given $g_1,g_2,g_3\in Homeo_{+}(\R)$ we consider the equation 
$$ (E)\text{ }\text{ } Xg_1Xg_2Xg_3=id $$ 
in the group $Homeo_{+}(\R)$. 
A partial solution is a homeomorphism \\
$h:(-\infty,p)\to(-\infty,q)$ such that $hg_1hg_2hg_3(x)=x$ for every $x$ in the domain of the composition. 
\end{defi}

\begin{rem} \label{borde} We do not exclude $+\infty$ as a value for $p$ or $q$ in the definition of partial solution. We use the notation $h(p)=q$.
\end{rem}

If $h$ is a partial solution of the equation $(E)$ then the domain of the composition $hg_1hg_2hg_3$ is an interval of the form $(-\infty,e_h)$.

\begin{rem} If $h:(-\infty,p)\to (-\infty,h(p))$ is a partial solution  of $(E)$
and \\ $h_1:(-\infty,p_1)\to(-\infty,h_1(p_1))$ is a partial solution extending $h$ (where $p_1>p$), then $e_{h_1}>e_h$.
\end{rem}

We will need a notion of trajectories for partial solutions of the equation $(E)$. It will also be important to look at the largest point of a trajectory where we apply the partial solution. As we did in the context of representations, we will build partial (and total) solutions via pre-fixing trajectories.

\begin{defi} Given the equation $(E)$, a partial solution $h$ and $x<e_h$, we define a sequence 
$$S(E,h,x)=(x,g_3(x),hg_3(x),g_2hg_3(x),hg_2hg_3(x),g_1hg_2hg_3(x),hg_1hg_2hg_3(x)=x)$$ 

In turn, for any sequence of the form $S=(a_1,b_1,a_2,b_2,a_3,b_3,a_1)\in\R^7$ we define $d_X S=\max\{b_1,b_2,b_3\}$. 
\end{defi}

These definitions are very similar to the definition of a trajectory for a representation, given in section \ref{troncocomun}, and to Definition \ref{trajectory}. In fact, when $\rho\in \text{Rep}(G,Homeo_{+}(\R))$ and $g_j=\rho(w_j)$, we have that $\rho(t)$ is a solution of $(E)$ and (setting $t=x_{n+1}$, and recalling that $\bar w=tw_1tw_2tw_3$) we get $d_{\bar w}(S(\rho,\bar w,x),n+1)=d_XS(E,\rho(t),x)$. 

Notice that $d_X(E,h,x)$ tends to $p$ as $x$ tends to $e_h$, by our definition of $e_h$.

\begin{lem} \label{acotadas} If $h:(-\infty,p)\to(-\infty,h(p))$ is a partial solution of $(E)$ with $p<+\infty$, then $h(p)<+\infty$. 
\end{lem}

\begin{proof}
Since $h$ must be defined on $(-\infty,g_3(e_h))$, we have $e_h\leq g_3^{-1}(p)<+\infty$. Notice that every entry in $S(E,h,x)=(a_1(x),b_1(x),a_2(x),b_2(x),a_3(x),b_3(x),a_1(x))$ is increasing on $x\in(-\infty,e_h)$. Also that $b_j(x)<p$ for every $j$ and $x<e_h$, since $h$ must be defined on $b_j(x)$. Recalling that  $d_XS(E,h,x)\to p$ as $x\to e_h$ we see there is some $i$ so that $b_i(x)\to p$ when $x\to e_h$. If $h(p)=+\infty$ then $a_{i+1}(x)=h(b_i(x))\to +\infty$ as $x\to e_h$. But then $b_{i+1}(x)=g_{i+1}(a_{i+1}(x))\to+\infty$ which is absurd since $b_j(x)<p$ for every $j$.

\end{proof}

\begin{figure}[h!] 
\includegraphics[width=1\textwidth]{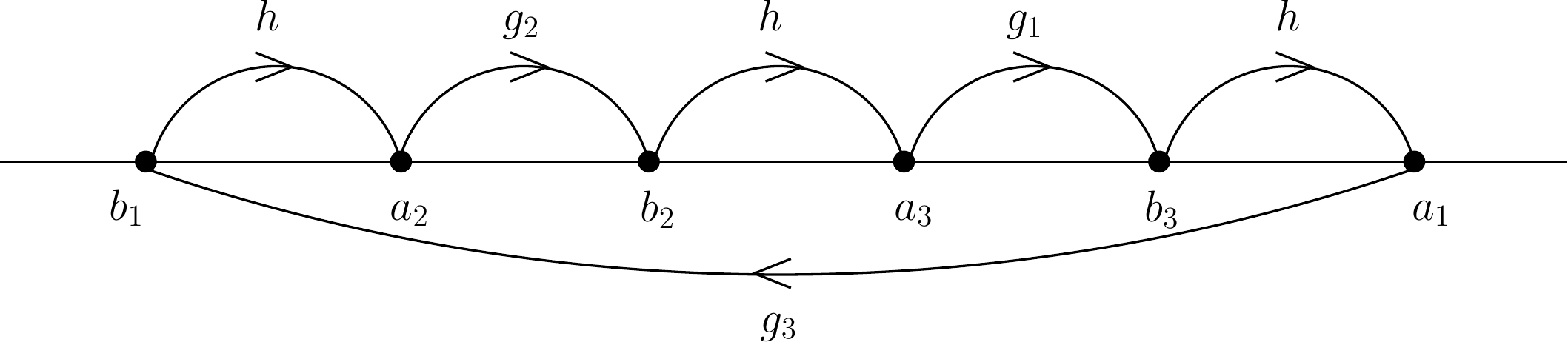}
\caption{This shows a possible sequence $S(E,h,x)$. It is an instance of Case A in Lemma \ref{extension}.}
\label{f5}
\end{figure}

In light of Lemma \ref{acotadas} we can assume that partial solutions are defined on $(-\infty,p]$, and so we have $d_XS(E,h,e_h)=p$.

We consider the cyclic permutations of the equation $(E)$:
\begin{itemize} 
\item $(E')$ $\ \ $ $Xg_3Xg_1Xg_2=id$
\item $(E'')$ $\ \ $ $Xg_2Xg_3Xg_1=id$
\end{itemize}

Observe that $(E)$, $(E')$ and $(E'')$ share the same partial solutions, though the domains of the composition may be different.
 
\begin{lem}\label{extension}If $h:(-\infty,p]\to(-\infty,h(p)]$ is a partial solution of $(E)$ and there exists $i,j$ such that $g_i(e_h)\neq g_j(e_h)$ then we can extend $h$ to a partial solution $h'$ defined on $(-\infty,p+\epsilon]$ for some $\epsilon>0$.
\end{lem}

\begin{proof} Let $S(E,h,e_h)=(a_1,b_1,a_2,b_2,a_3,b_3,a_1)$. The condition that $g_i(e_h)\neq g_j(e_h)$ for some $i,j$ rules out that $b_1=b_2=b_3$. That is because $a_{i+1}=h(b_i)$ (where the indices are taken mod 3), so if $b_1=b_2=b_3=b$ then $b=p$ and $a_1=a_2=a_3=h(p)=e_h$, thus we would have $g_1(e_h)=g_2(e_h)=g_3(e_h)=p$.

 Hence we must have one of the following two cases (taking indices mod $3$):

\begin{itemize}
\item Case A: $b_j>b_{j\pm1}$ for some $j$.
\item Case B: $b_j=b_{j-1}>b_{j+1}$ for some $j$.
\end{itemize}

Possibly exchanging $(E)$ by a suitable cyclic permutation we can assume that $j=3$. Thus $b_3=d_XS(E,h,e_h)=p$.
\vs

Case A: Consider the map $\psi=g_1hg_2hg_3$ where this composition makes sense. It certainly is defined at $a_1=e_h$, with $\psi(e_h)=p$. Since the $g_i$ are defined on $\R$ and $h$ is defined up to $p>b_1,b_2$, then this composition is also defined up to $e_h+\delta$ for $\delta>0$. Then $h'=\psi^{-1}$ is defined on $(-\infty,p+\epsilon)$ for some $\epsilon>0$. It agrees with $h$ on $(-\infty,p]$ since $h$ satisfies $(E)$. It is also a partial solution of $(E)$: if $x<e_{h'}$, then $h'g_1h'g_2h'g_3(x)=\psi^{-1}g_1hg_2hg_3(x)=x$.

Notice that in this case the local extension $h'$ on $(-\infty,p+\epsilon)$ is unique.

\vs
Case B: Since $b_3=b_2$ we get $a_1=a_3$ by applying $h$, thus $a_1=e_h$ is fixed by $hg_1$. Writing $\varphi_0=hg_1$, which is defined on $(-\infty,e_h]$, we get that $\varphi_0^2 g_1^{-1}g_2hg_3(x)=x$ for $x<e_h$. If we take $\psi= g_1^{-1}g_2hg_3$, we see that it is defined on $(-\infty,e_h+\delta]$ for some $\delta>0$, since $b_1<p$. Let $\varphi$ be a square root of $\psi^{-1}$ that agrees with $\varphi_0$ up to $a_1$. This map exists by Lemma \ref{raiz} and Remark \ref{mas raiz}, since $a_1$ is fixed. Then $h'=\varphi g_1^{-1}$ is the desired extension, which is defined on $(-\infty,p+\epsilon]$ for some $\epsilon>0$.

\end{proof}

\begin{figure}[h!] 
\centering
\includegraphics[width=0.8\textwidth]{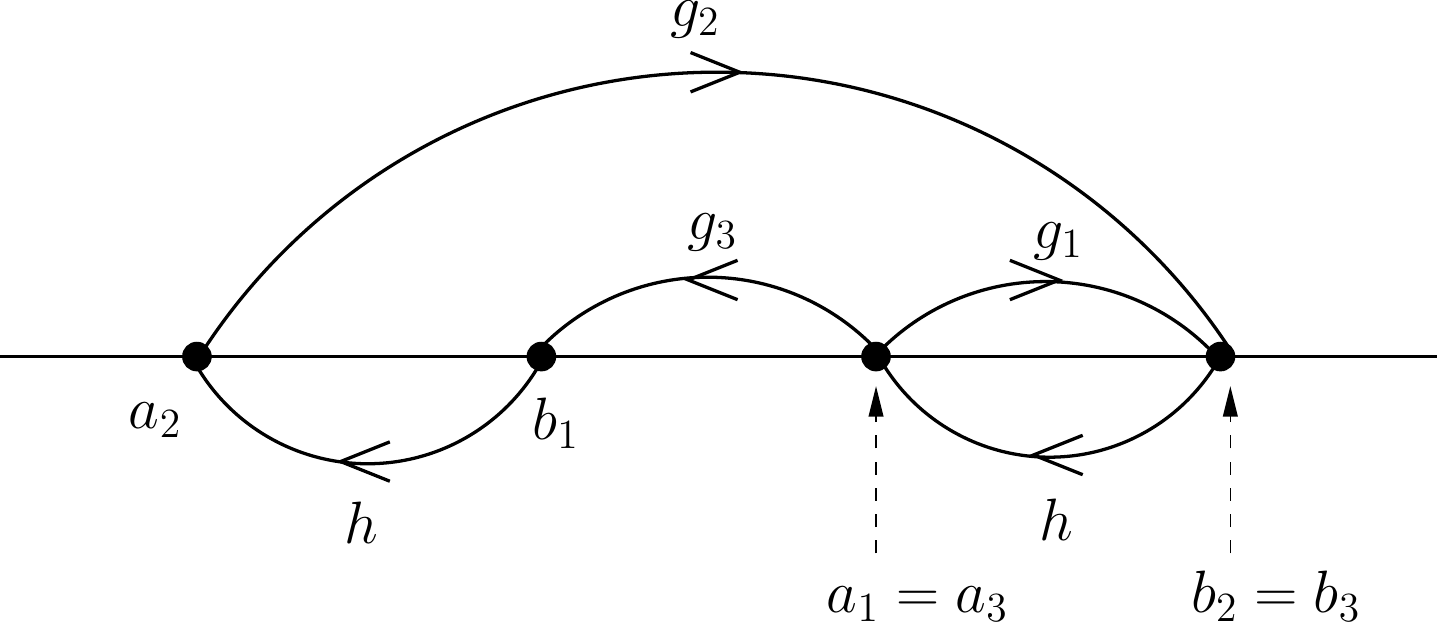}
\caption{This shows an instance of Case B in Lemma \ref{extension}.}
\label{f6}
\end{figure}

\begin{lem} \label{globalema} If $h:(-\infty,p)\to(-\infty,h(p))$ is a partial solution of $(E)$ and there exists $i,j$ such that $g_i(x)\neq g_j(x)$ for all $x\geq e_h$ then we can extend $h$ to a solution $h'\in Homeo_{+}(\R)$.
\end{lem}
\begin{proof} A maximal extension of $h$ is an homeomorphism of the line by Lemmas \ref{acotadas} and \ref{extension}.
\end{proof}

\newpage 

{\bf Proof of Theorem \ref{positiveword}} 
\vs

Let $<$ be a left order on $G$ with dynamical realization $\rho$. We will repeat the same strategy of the previous theorems, that is to construct a small perturbation of $\rho$ that has non trivial stabilizer on the orbit of $0$.

Take an arbitrarily large compact interval $K\subseteq \R$. As in the previous theorems, we can find a finite increasing $\rho$-trajectory $S$, with respect to the generating set $\{t,x_1,...,x_n\}$, that begins at $0$ and ends outside $K$. Recall that each point of $S$ is obtained from the previous one by acting either by $\rho(t)^{\pm 1}$ or by $\rho(x_i)^{\pm 1}$ for some $i\in\{1,\ldots,n\}$. As in Theorem \ref{hnn}, write $S=\bigcup_{j=1}^k S(\rho|_{\F_n},v_j,p_j)$, where $p_1=0$ and $p_{j+1}$ is obtained by acting on $\rho(v_j)(p_j)$ by $\rho(t)^{\pm 1}$. (If $v_j$ is trivial, then $S(\rho|_{\F_n},v_j,p_j)$ is just the point $p_j$).

Let $q=\rho(v_k)(p_k)$ be the last element of $S$.

Choosing $S$ so that $q$ is large enough, we can assume that $\min S(\rho,\bar w,q)>\max K$, where we recall that $\bar w=tw_1tw_2tw_3$. We consider, for $i=1,\ldots,n$:
\begin{itemize}
\item $d_{ij} = d_{v_j}(S(\rho,v_j,p_j),i) \text{ for } j=1,\ldots,k$
\item $d_i^{(3)}=d_{w_3}( S(\rho,w_3,q_3),i)$, where  $q_3=q$.
\item $d_i^{(2)}=d_{w_2}( S(\rho,w_2,q_2),i)$, where  $q_2=\rho(tw_3)(q)$.
\item $d_i^{(1)}=d_{w_1}( S(\rho,w_1,q_1),i)$, where $q_1=\rho(tw_2tw_3)(q)$.

\end{itemize}

By Lemma \ref{increasing} there exists $i_0$ with $d_{i_0}^{(3)}\geq\max_j d_{i_0j}$. Let $l\in\{1,2,3\}$ be such that $d_{i_0}^{(l)}=\max\{d_{i_0}^{(1)},d_{i_0}^{(2)},d_{i_0}^{(3)}\}$.

Let $w=w_{l+1}^{-1}w_l$ where $l+1$ is taken mod 3. Writing $d_i=d_w(S(\rho,w,q_l),i)$, we have that $d^{(l)}_{i_0}\leq d_{i_0}$ since $w_l$ is the final segment of $w$. That is because the string $w_{l+1}^{-1}w_l$ is reduced, which can be assumed by Lemma \ref{lasmagias}. 


We apply Lemma \ref{tronco} to $\rho|_{\F_n}$, setting:
\begin{itemize}
\item $v_1,\ldots,v_k$ as given above, $v_{k+1}=w_1$, $v_{k+2}=w_2$, $v_{k+3}=w_3$, and $w=w_{l+1}^{-1}w_l$ (also as above).
\item $p_1\ldots,p_k$ as given above, $p_{k+1}=q_1$, $p_{k+2}=q_2$, $p_{k+3}=q_3$, and $p=q_l$.
\end{itemize}
The hypotheses of Lemma \ref{tronco} hold by the previous discussion. Let $\rho'\in\text{Rep}(\F_n,Homeo_{+}(\R)$ be the representation given by Lemma \ref{tronco}.



Let $g_j=\rho'(w_j)$ for $j=1,2,3$. We consider the equation 
$$(E) \ \ \ Xg_1Xg_2Xg_3=id $$

Point 1 in Lemma \ref{tronco} guarantees that $g_j$ agrees with $\rho(w_j)$ up to $q_j$. Thus $\rho(t)$ is a partial solution of $(E)$ when restricted to $(-\infty,z)$, where 
$z=\max\{g_1(q_1),g_2(q_2),g_3(q_3)\}$.  


Let $h=\rho(t)|_{(-\infty,z)}$.  We consider a cyclic permutation of $(E)$ so that $e_h=q_l$ (for $(E)$ itself we have $e_h=q$).

By point 3 in Lemma \ref{tronco} we have $g_l(x)\neq g_{l+1}(x)$ for all $x\geq q_l$. This allows us to apply Lemma \ref{globalema}, so $h$ can be extended to a maximal solution $\psi\in Homeo_{+}(\R)$. This gives rise to a representation $\bar \rho \in\text{Rep}(G,Homeo_{+}(\R))$ with $\bar \rho|_{\F_n}=\rho'$ and $\bar \rho(t)=\psi$. 

This $\bar\rho$ is a small perturbation of $\rho$ because of point 1 in Lemma \ref{tronco} and the fact that $\psi$ agrees with $\rho(t)$ on $(-\infty,z]$. Also by this fact and point 2 in Lemma  \ref{tronco} we see that neither $S$ nor $S(\rho,\bar w,q)$ are affected by this perturbation, so $p=q_l$ is in the orbit of $0$ under $\bar\rho$. Then we use point 4 in Lemma \ref{tronco} to conclude that the stabilizer of $0$ under $\bar\rho$ is non trivial. Thus Proposition \ref{aislado} finishes the proof.

\begin{small}

\textit{Juan Alonso}

Fac. Ciencias, Universidad de la Republica Uruguay

juan@cmat.edu.uy

\bigskip

\textit{Joaquin Brum}

Fac. Ingenieria, Universidad de la Republica Uruguay

joaquinbrum@fing.edu.uy

\end{small}

\end{document}